\documentclass[reqno]{amsart}
\usepackage{amssymb}
\usepackage[mathscr]{eucal}
\usepackage{hyperref}
\usepackage{graphicx}
\usepackage{xcolor}
\usepackage[all]{xy}

\theoremstyle{plain}
\newtheorem{prop}{Proposition}[section]
\newtheorem{thm}[prop]{Theorem}
\newtheorem{cor}[prop]{Corollary}

\newtheorem{lem}[prop]{Lemma}

\theoremstyle{definition}
\newtheorem{dfn}[prop]{Definition}
\newtheorem{rem}[prop]{Remark}

\newtheorem{example}[prop]{Example}

\newtheorem{lab}[prop]{}

\newcommand{\isoto}{\overset{\sim}{\to}}

\renewcommand{\iff}{\Leftrightarrow}

\newcommand{\A}{{\mathbb{A}}}
\newcommand{\C}{{\mathbb{C}}}
\newcommand{\bbH}{{\mathbb{H}}}
\newcommand{\K}{{\mathbb{K}}}

\newcommand{\Q}{{\mathbb{Q}}}
\newcommand{\R}{{\mathbb{R}}}

\renewcommand{\a}{{\mathfrak{a}}}
\newcommand{\g}{{\mathfrak{g}}}
\newcommand{\h}{{\mathfrak{h}}}
\renewcommand{\k}{{\mathfrak{k}}}

\newcommand{\p}{{\mathfrak{p}}}

\renewcommand{\t}{{\mathfrak{t}}}

\newcommand{\E}{\mathsf{\Lambda}}

\newcommand{\scrO}{{\mathscr{O}}}

\DeclareMathOperator{\Ad}{Ad}
\DeclareMathOperator{\Aut}{Aut}

\DeclareMathOperator{\conv}{conv}
\DeclareMathOperator{\End}{End}
\DeclareMathOperator{\Hom}{Hom}

\DeclareMathOperator{\tr}{tr}
\DeclareMathOperator{\trd}{trd}

\DeclareTextFontCommand{\textnf}{\normalfont}

\newcommand{\cone}{\mathrm{cone}}
\newcommand{\diag}{\mathrm{diag}}

\newcommand{\id}{\mathrm{id}}

\newcommand{\relint}{\mathrm{relint}}

\newcommand{\x}{{\mathtt{x}}}
\newcommand{\y}{{\mathtt{y}}}

\newcommand{\du}{{\scriptscriptstyle\vee}}

\renewcommand{\emptyset}{\varnothing}
\renewcommand{\setminus}{\smallsetminus}
\newcommand{\hvar}{{-}}

\newcommand{\ol}{\overline}
\newcommand{\plus}{{\scriptscriptstyle+}}
\newcommand{\minus}{{\scriptscriptstyle-}}
\newcommand{\wh}[1]{\widehat{#1}}
\newcommand{\wt}[1]{\widetilde{#1}}
\newcommand{\all}{\forall\,}

\renewcommand{\subset}{\subseteq}

\newcommand{\nml}{\trianglelefteq}

\renewcommand{\choose}[2]{\genfrac(){0pt}{}{#1}{#2}}
\newcommand{\bil}[2]{\langle{#1},{#2}\rangle}

\begin{document}

\title
[Spectrahedral representation of polar orbitopes]
{Spectrahedral representation of polar orbitopes}

\begin{abstract}
Let $K$ be a compact Lie group and $V$ a finite-dimensional
representation of $K$. The orbitope of a vector $x\in V$ is the
convex hull $\scrO_x$ of the orbit $Kx$ in $V$. We show that if $V$
is polar then $\scrO_x$ is a spectrahedron, and we produce an
explicit linear matrix inequality representation. We
also consider the coorbitope $\scrO_x^o$, which is the convex set
polar to $\scrO_x$. We prove that $\scrO_x^o$ is the convex hull of
finitely many $K$-orbits, and we identify the cases in which
$\scrO_x^o$ is itself an orbitope. In these cases one has
$\scrO_x^o=c\cdot\scrO_x$ with $c>0$. Moreover we show that if $x$
has ``rational coefficients'' then $\scrO_x^o$ is again a
spectrahedron. This provides many new families of doubly
spectrahedral orbitopes. All polar orbitopes that are derived from
classical semisimple Lie can be described in terms of conditions on
singular values and Ky Fan matrix norms.
\end{abstract}

\author{Tim Kobert, Claus Scheiderer}
\address
  {Fachbereich Mathematik und Statistik \\
  Universit\"at Konstanz \\
  78457 Konstanz \\
  Germany}
\email
  {tim.kobert@uni-konstanz.de, claus.scheiderer@uni-konstanz.de}



\thanks
{This work was partially supported by DFG grants SCHE281/10-1 and
SCHE281/10-2}

\date\today
\maketitle


\section*{Introduction}

Let $K$ be a compact Lie group, and let $V$ be a finite-dimensional
real representation of $K$. The orbitope of a vector $x\in V$,
denoted $\scrO_x$, is the convex hull of the orbit $Kx$ in $V$.
Orbitopes are highly symmetric objects that are interesting from many
perspectives, like convex geometry, algebraic geometry, Lie theory,
symplectic geometry, combinatorial geometry or optimization.
We refer to \cite{sss} for a broad overview with plenty of explicit
examples.

Here our focus will be on properties of orbitopes that are
particularly relevant to optimization, and more specifically, to
semidefinite programming. We are interested in existence and explicit
construction of spectrahedral representations for orbitopes and
related convex bodies. For this we consider a particular class of
group representations, namely polar representations of connected
compact Lie groups. As far as the orbit structure is concerned, all
such representations arise from Riemannian symmetric spaces $M=G/K$
as actions of the isotropy group on the tangent space at a point.
In other words, each polar representation comes from a Cartan
decomposition $\g=\k\oplus\p$ of a real semisimple Lie algebra $\g$,
as the adjoint representation of $K$ on $\p$. Dadok \cite{da} showed
that these representations have very particular properties. At the
same time they comprise all the familiar actions of the classical
unitary, orthogonal or symplectic groups on (skew-) hermitian or
symmetric matrices.

Our main results are as follows. We prove that every orbitope
$\scrO_x$ in a polar representation is a spectrahedron, i.e.\ an
affine-linear slice of the psd matrix cone. In fact we produce an
explicit linear matrix inequality representation for any such
orbitope (Theorem \ref{polorbspekt}). So far, this result was known
only for a few scattered classes of examples. We also consider the
dual convex body $\scrO_x^o$, called the coorbitope of~$x$. We prove
that $\scrO_x^o$ always is the convex hull of finitely many
$K$-orbits, and we identify those orbits explicitly (Corollary
\ref{orbitsoxo}). In particular, we isolate the cases when $\scrO_x$
is a biorbitope (Theorem \ref{biorbitop}), meaning that $\scrO_x^o$
is an orbitope as well. Remarkably, $\scrO_x$ is always self-polar up
to positive scaling when it is a biorbitope (Theorem \ref{oxovsox}).
Moreover, we show that whenever the orbitope $\scrO_x$ can be
``defined over the rational numbers~$\Q$'', the coorbitope
$\scrO_x^o$ is again a spectrahedron, and we find an explicit linear
matrix inequality for it (Theorem \ref{doublspect}). So far, only
very few examples of doubly spectrahedral sets were known
\cite{spw}. Our result provides many new series of sets with this
property.

The main and all-important tool for our results is Kostant's
convexity theorem \cite{ko}. It allows to reduce most questions
considered here to a Cartan subspace, and even to a Weyl chamber. In
this way the questions become polyhedral in nature.

The paper is organized as follows. Polar representations and
orbitopes are recalled in Section~1. The general background on
semisimple real Lie algebras and their restricted root systems is
summarized in Section~2, as far as it will be needed here. Kostant's
theorem
is stated in Section~3, together with a few immediate consequences.
Spectrahedral representations for polar orbitopes are constructed in
Section~4. In Sections 5 and~6 we relate the facial structure of
$\scrO_x$ to the momentum polytope $P_x$, with an emphasis on maximal
faces of $\scrO_x$ resp.\ facets of $P_x$. Since the maximal faces of
$\scrO_x$ correspond to the extreme points of the polar set
$\scrO_x^o$,
this allows us to identify biorbitopes. Doubly spectrahedral
orbitopes are considered in Section~7. Finally, in Section~8 we list
all polar orbitopes that are derived from semisimple Lie algebras of
classical type. All these orbitopes have descriptions in terms of
singular values of matrices over $\R$, $\C$ or $\bbH$. Typically,
they consist of intersections of balls of various radii with respect
to different Ky Fan matrix norms.

Some of the results presented here are taken from the 2018 doctoral
thesis \cite{kob} of the first author, written under the guidance of
the second.


\section{Polar representations and orbitopes}

\begin{lab}\label{polar}%
We recall the notion of polar representation, following Dadok
\cite{da}. Let $K$ be a Lie group with Lie algebra $\k$, and let
$K\to O(V)$ be a linear representation of $K$ on a
(finite-dimensional) real vector space $V$, preserving a fixed inner
product.
For every $x\in V$, the linear subspace $\a_x=(\k x)^\bot$ of $V$
meets every $K$-orbit \cite[Lemma~1]{da}. A vector $x\in V$ is said
to be regular if the orbit $Kx$ has maximal dimension.
The subspaces $\a_x$, for $x$ regular, are called the \emph{Cartan
subspaces} of the representation. The Cartan subspaces are all
$K$-conjugate if, and only if, they are orthogonal to the $K$-orbits
passing through them \cite[Prop.~2]{da}.
The representation $\rho$ is said to be \emph{polar} if these
equivalent conditions are satisfied.

Every Riemannian symmetric space $M$ gives rise to a polar
representation, namely the action of the isotropy group on the
tangent space $T_e(M)$ at a point $e\in M$. In other words, let $G$
be a connected real semisimple Lie group with Lie algebra $\g$, let
$\g=\k\oplus\p$ be a Cartan decomposition, and let $K\subset G$ be
the analytic subgroup corresponding to $\k$. Then $K$ is a maximal
compact subgroup of~$G$, and the adjoint action of $K$ on $\p$ is an
example of a polar representation. This representation is irreducible
if and only if $\g$ is simple as a Lie algebra.
Conversely, as far as the orbit structure is concerned, these are the
only examples of polar representations of connected Lie groups:
\end{lab}

\begin{prop}\label{dadokprop6}%
\emph{(Dadok \cite[Proposition~6]{da})}
Let $V$ be a polar representation of a connected Lie group $H$. There
is a real semisimple Lie algebra $\g$ with Cartan decomposition
$\g=\k\oplus\p$, together with a vector space isomorphism
$f\colon V\to\p$, such that $f(H\cdot x)=\Ad(K)\cdot f(x)$ for every
$x\in V$, where $K\subset\Aut(\g)$ is the analytic subgroup with Lie
algebra~$\k$.
\end{prop}

\begin{rem}\label{dirprod}%
Let $V=\bigoplus_{i=1}^nV_i$ be the irreducible decomposition of an
arbitrary polar representation of $K$. Then each irreducible summand
$V_i$ is again a polar representation of $K$ (Dadok
\cite[Theorem~4]{da}). Moreover, if $x=\sum_{i=1}^nx_i\in V$ with
$x_i\in V_i$ for all~$i$, then $\scrO_x=\scrO_{x_1}\times\cdots\times
\scrO_{x_n}$ with $\scrO_{x_i}=\conv_{V_i}(Kx_i)$. Indeed, such a
direct product decomposition holds for the $K$-orbit of~$x$ by
Dadok's theorem, and hence it holds for the convex hulls as well.
In other words, every polar orbitope is a cartesian direct product of
irreducible polar orbitopes.
\end{rem}

\begin{lab}
Recall that a \emph{spectrahedron} in $\R^n$ is the solution set of a
linear matrix inequality (LMI). So $S\subset\R^n$ is a spectrahedron
if there exist complex hermitian matrices $A_0,\dots,A_n$ of some size
$d\times d$ such that
$$S\>=\>\Bigl\{x\in\R^n\colon A_0+\sum_{i=1}^nx_iA_i\succeq0\Bigr\},$$
where $A\succeq0$ means that $A$ is positive semidefinite (all
eigenvalues are nonnegative). Note that an LMI with complex hermitian
$d\times d$ matrices may be converted into an equivalent LMI with
real symmetric $2d\times2d$ matrices, which is why spectrahedra are
often defined via real symmetric LMIs.
\end{lab}

\begin{lab}\label{orbitopsasum}%
Let $V$ be a linear representation (real and finite-dimensional) of a
compact Lie group $K$. Given $x\in V$, the convex hull $\conv(Kx)$ of
the orbit of $x$ in $V$ is called the ($K$-) \emph{orbitope} of $x$.
We usually denote it by $\scrO_x=\conv(Kx)$, assuming that $K$ and
$V$ are understood. The orbitope $\scrO_x$ is a compact convex set on
which $K$ acts, and whose set of extreme points coincide with the
orbit~$Kx$.

We will study orbitopes $\scrO_x=\conv(Kx)$ in polar representations
$V$ of compact connected Lie groups $K$. Using Proposition
\ref{dadokprop6}, we can and will always assume that $V=\p$ where
$\g=\k\oplus\p$ is a Cartan decomposition of a real semisimple Lie
algebra $\g$, and that the action is the adjoint action of the
analytic subgroup $K$ of $\Aut(\g)$.
\end{lab}

\begin{lab}
If $V$ is a vector space over $\R$, the dual vector space is denoted
by $V^\du=\Hom(V,\R)$. The convex hull of a set $M\subset V$ is
written $\conv(M)$. Our notation for matrix groups and matrix Lie
algebras tries to follow the conventions in \cite{kn}.
In particular, $SU(n)$, $SO(n)$, $Sp(n)$ are the classical compact
Lie groups, $su(n)$, $so(n)$, $sp(n)$ are their Lie algebras, etc.
The diagonal $n\times n$ matrix with diagonal entries $a_1,\dots,a_n$
is denoted $\diag(a_1,\dots,a_n)$.
\end{lab}


\section{Background on semisimple real Lie algebras}%
\label{lienot}%

We use standard notation and terminology for semisimple Lie groups
and Lie algebras, and we'll recall it briefly here. As a general
reference
we refer to Knapp's monograph \cite{kn}, in particular to Chapter~6.

\begin{lab}
Let $\g$ be a semisimple Lie algebra over $\R$, hence a finite direct
sum of simple (nonabelian) Lie algebras over $\R$. Recall that if
$\g$ is simple then either $\g$ has a structure as a (simple) Lie
algebra over~$\C$, or else $\g_\C:=\g\otimes\C$ is a simple Lie
algebra over~$\C$.

Let $\theta$ be a Cartan involution on $\g$ and $\g=\k\oplus\p$ the
corresponding Cartan decomposition.
With respect to the Killing form $\bil\hvar\hvar$ of $\g$, the
decomposition $\k\oplus\p$ is orthogonal, and the restriction of
$\bil\hvar\hvar$ to $\k$ (resp.~$\p$) is negative (resp.\ positive)
definite.
\end{lab}

\begin{lab}
Choose a maximal commutative subspace $\a$ of $\p$, and let
$\Sigma\subset\a^\du=\Hom(\a,\R)$ be the system of restricted roots
of $(\g,\a)$. This is an abstract root system, possibly non-reduced.
Fixing an ordering on $\a^\du$ we have the sets
$\Sigma_\plus\subset\Sigma$ of positive restricted roots and
$\Gamma=\{\beta_1,\dots,\beta_n\}\subset\Sigma_\plus$ of simple
restricted roots.
\end{lab}

\begin{lab}\label{comporderings}%
Let $\t$ be a maximal commutative subalgebra of the centralizer
$Z_\k(\a)$ of $\a$ in $\k$. Then $\h:=i\t\oplus\a$ is a commutative
subalgebra of $\g_\C$ (the complexification of $\g$) for which
$\h_\C$ is a Cartan algebra of $\g_\C$. The rank of $\g_\C$ will be
denoted by $l=\dim(\h)$, the real rank of $\g$ by $n=\dim(\a)$.

The Killing form of $\g_\C$ restricted to $\h$ is a euclidean inner
product on $\h$, denoted by $\bil\hvar\hvar$. All roots of
$(\g_\C,\h_\C)$ are real-valued on $\h$. We let
$\Delta\subset\h^\du=\Hom(\h,\R)$ be the root system of
$(\g_\C,\h_\C)$.
Extending the ordering on $\a^\du$ suitably to $\h^\du$ (e.g.\ take
the lexicographic order on $\h^\du=(\a+i\t)^\du$ with $\a$ before
$i\t$ \cite[p.~377]{kn}),
we let $\Delta_\plus\subset\Delta$ be the positive roots and
$\Pi=\{\alpha_1,\dots,\alpha_l\}\subset\Delta_\plus$ the simple
roots.
\end{lab}

\begin{lab}
Let $G=\Aut(\g)_0$, the identity component of the automorphism group
of $\g$. The Lie algebra of $G$ is naturally identified with $\g$.
The analytic subgroup $K\subset G$ with Lie algebra $\k$ is a maximal
compact subgroup of~$G$, and $K$ acts on ($\k$ and) $\p$ via the
adjoint action. We denote this action simply by $gx:=\Ad(g)x$, for
$g\in K$ and $x\in\p$. Every element of $\p$ is $K$-conjugate to an
element of~$\a$.
Note that the $K$-action preserves the quadratic (Killing) form
on $\p$. This action is a polar representation of $K$, and as far as
the orbit structure is concerned, every polar representation of a
connected Lie group arises in this way (Proposition
\ref{dadokprop6}).
\end{lab}

\begin{lab}
Note that the direct sum decomposition $\h=i\t\oplus\a$ is orthogonal
with respect to the Killing form. The restriction map
$r\colon\h^\du\to\a^\du$ satisfies $r(\Delta_\plus)\subset
\Sigma_\plus\cup\{0\}$ and $r(\Pi)\subset \Gamma\cup\{0\}$.
Conversely, there exists an involution $\alpha\mapsto\alpha'$ on
$\Pi$ such that for every $\beta\in\Gamma$, the set
$\{\alpha\in\Pi\colon r(\alpha)=\beta\}$ has the form
$\{\alpha,\alpha'\}$ with $\alpha\in\Pi$.
\end{lab}

\begin{lab}\label{weyl}%
From the inner product $\bil\hvar\hvar$ on $\h$ we get linear
isomorphisms $\h\isoto\h^\du$ and $\a\isoto\a^\du$. We use them to
transfer the inner product from $\h$ to $\h^\du$ and from $\a$ to
$\a^\du$.

For every restricted root $\beta\in\Sigma$ let $s_\beta\colon
\a^\du\to\a^\du$ denote the root reflection $\gamma\mapsto
\gamma-\frac{2\bil\beta\gamma}{|\beta|^2}\beta$. We always write
$W=\langle s_\beta\colon\beta\in\Sigma\rangle$ for the (restricted)
Weyl group of $(\g,\a)$.
Via the identification $\a\isoto\a^\du$ we consider $W$ as a
reflection group on $\a$ as well: For $\beta\in\Sigma$, if
$h_\beta\in\a$
is the element with $\bil{h_\beta}x=\beta(x)$ for all
$x\in\a$, then $s_\beta$ acts on $\a$ by
$s_\beta(x)=x-\frac{2\beta(x)}{|h_\beta|^2}h_\beta$. It is well known
that $W$ is naturally isomorphic to $N_K(\a)/Z_K(\a)$
\cite[6.57]{kn}.

We always denote by
$$C\>=\>\bigl\{x\in\a\colon\beta_1(x)\ge0,\dots,\beta_n(x)\ge0
\bigr\}$$
the (closed) Weyl chamber for the action of $W$. So $C$ is a
polyhedral convex cone, and every element of $\a$ is $W$-conjugate to
a unique element of $C$.
\end{lab}

\begin{lab}\label{dualbas}%
Let $\lambda_1,\dots,\lambda_l$ be the basis of $\h^\du$ that is dual
to $\Pi=\{\alpha_1,\dots,\alpha_l\}$, so $\bil{\alpha_i}{\lambda_k}=
\delta_{ik}$ for $i,k=1,\dots,l$. Similarly, let $\mu_1,\dots,\mu_n
\in\a^\du$ be defined by $\bil{\beta_j}{\mu_k}=\delta_{jk}$ for
$j,k=1,\dots,n$.
The fundamental weights of $(\g_\C,\h_\C)$ are the linear forms
$\omega_i\in\h^\du$ defined by $\omega_i=\frac12|\alpha_i|^2
\lambda_i$ ($i=1,\dots,l$), so $\frac{2\bil{\omega_i}{\alpha_k}}
{|\alpha_k|^2}=\delta_{ik}$ for $i,k=1,\dots,l$. For each linear
combination $\omega=\sum_{i=1}^lm_i\omega_i$ with integer
coefficients $m_i\ge0$, there exists a unique (up to isomorphism)
irreducible representation of $\g_\C$ with highest weight $\omega$.
The irreducible representation of $\g_\C$ with highest weight
$\omega_i$ is called the $i$-th fundamental representation of
$\g_\C$, we'll denote it by~$\rho_i$  ($i=1,\dots,l$).
\end{lab}

\begin{lab}
As before, let $\Pi\subset\Delta_\plus\subset\h^\du$ resp.\
$\Gamma\subset\Sigma_\plus\subset\a^\du$ be the systems of simple
resp.\ simple restricted roots. We need to relate the dual bases of
$\Pi$ and $\Gamma$ to each other. Given $\alpha\in\Pi$, let
$\lambda_\alpha\in\h^\du$ be defined by $\bil{\lambda_\alpha}
{\alpha'}=\delta_{\alpha,\alpha'}$ for each $\alpha'\in\Pi$. Given
$\beta\in\Gamma$, let $\mu_\beta\in\a^\du$ be defined by
$\bil{\mu_\beta}{\beta'}=\delta_{\beta,\beta'}$ for each
$\beta'\in\Gamma$. (So if $\alpha=\alpha_i$ then
$\lambda_\alpha=\lambda_i$, and similarly if $\beta=\beta_j$ then
$\mu_\beta=\mu_j$.) The following fact is certainly well-known, but
we haven't been able to find a suitable reference for it:
\end{lab}

\begin{lem}\label{restrfundgew}%
Let $\alpha\in\Pi$. If $\beta:=r(\alpha)\ne0$ then
$r(\lambda_\alpha)=q\mu_\beta$ for some rational number $q>0$. If
$r(\alpha)=0$ then $r(\lambda_\alpha)=0$.
\end{lem}

In fact the argument shows that $r(\lambda_\alpha)=\mu_\beta$ if
$\alpha$ is the only element of $\Pi$ that restricts to $\beta$, and
$r(\lambda_\alpha)=\frac12\mu_\beta$ if there are two such elements.

\begin{proof}
Let $\a^\du\to\h^\du$, $\mu\mapsto\wt\mu$ denote the linear map
that is adjoint to the restriction map $r\colon\h^\du\to\a^\du$, so
$\bil{\lambda}{\wt\mu}=\bil{r(\lambda)}\mu$ for $\lambda\in\h^\du$
and $\mu\in\a^\du$. Let $\Pi_0=\{\alpha\in\Pi$: $r(\alpha)=0\}$ and
$\Pi_1=\Pi\setminus\Pi_0$. The Cartan involution $\theta$ of $\g$
induces $\id$ on $i\t$ and $-\id$ on $\a$. The dual involution
$\theta^\du$ on $\h^\du$ satisfies $-\theta^\du(\Pi)\subset\Pi$. For
$\alpha\in\Pi$ we abbreviate $\alpha':=-\theta^\du(\alpha)$. If
$\alpha\in\Pi_1$, then the only elements of $\Pi$ that restrict to
$\beta=r(\alpha)$ are $\alpha$ and $\alpha'$ (see \cite{bu} Prop.\
29.9 for these facts).

Let $\beta\in\Gamma$ be a restricted simple root, and let
$\alpha\in\Pi$ with $r(\alpha)=\beta$. From
$\alpha'=-\theta^\du(\alpha)$ we see $\alpha+\alpha'=2\wt\beta$.
If $\lambda\in\h^\du$ is dominant, i.e.\ satisfies
$\bil\lambda\gamma\ge0$ for all $\gamma\in\Pi$,
we conclude $\bil{r(\lambda)}\beta=\bil\lambda{\wt\beta}=\frac12
\bil\lambda{\alpha+\alpha'}\ge0$. This shows
\begin{equation}\label{inclus}%
\cone(r(\lambda_1),\dots,r(\lambda_l))\>\subset\>
\cone(\mu_1,\dots,\mu_n)
\end{equation}
in $\a^\du$.

On the other hand, for $\gamma\in\Pi$ the inner product
$\bil{\wt{\mu_\beta}}\gamma=\bil{\mu_\beta}{r(\gamma)}$ is $1$ if
$r(\gamma)=\beta$, and $0$ otherwise.
Therefore $\wt{\mu_\beta}=\frac12(\lambda_\alpha+\lambda_{\alpha'})$.
In particular, $\mu_\beta=r(\wt\mu_\beta)=\frac12r(\lambda_\alpha+
\lambda_{\alpha'})$, which proves the reverse inclusion
of~\eqref{inclus}.

If $\alpha'=\alpha$ then $r(\lambda_\alpha)=\mu_\beta$. It remains
to consider the case $\alpha'\ne\alpha$.
From $\alpha+\alpha'=2\wt\beta$ we see $\bil{r(\lambda_\alpha)}\beta=
\bil{\lambda_\alpha}{\wt\beta}=\frac12$,
and similarly $\bil{r(\lambda_{\alpha'})}\beta=\frac12$. On the other
hand, $r(\lambda_\alpha+\lambda_{\alpha'})=2\mu_\beta$, together with
\eqref{inclus}, implies that $r(\lambda_\alpha)$,
$r(\lambda_{\alpha'})$ are positive scalar multiples of $\mu_\beta$,
since $\mu_\beta$ generates an extreme ray of
$\cone(\mu_1,\dots,\mu_m)$.
Altogether it follows that $r(\lambda_\alpha)=r(\lambda_{\alpha'})=
\frac12\mu_\beta$.

Finally assume $r(\alpha)=0$, and let $\gamma\in\Pi$ be arbitrary
with $\beta=r(\gamma)\ne0$. Then $\bil{r(\lambda_\alpha)}\beta=
\bil{\lambda_\alpha}{\wt\beta}=\frac12\bil{\lambda_\alpha}
{\gamma+\gamma'}$ since $2\wt\beta=\gamma+\gamma'$, see above. But
$\alpha\notin\{\gamma,\gamma'\}$ since $r(\gamma)=r(\gamma')=\beta
\ne0$. Hence $\bil{r(\lambda_\alpha)}\beta=0$. This for all
$\beta\in\Gamma$ shows $r(\lambda_\alpha)=0$.
\end{proof}


\section{Kostant's convexity theorem}

We assume the setup of Section~\ref{lienot}. So $\g$ is a semisimple
real Lie algebra with Cartan decomposition $\g=\k\oplus\p$ and
maximal abelian subspace $\a$ of $\p$. This gives us the system
$\Sigma\subset\a^\du$ of reduced roots of $(\g,\a)$, on which the
reduced Weyl group $W$ acts. After fixing an ordering we have the
simple positive roots $\Gamma$ and the Weyl chamber $C\subset\a$.

The key technical tool for this paper is Kostant's convexity
theorem, together with its consequences. To a large extent, it
allows to reduce the study of the polar orbits $Kx$ and their
orbitopes $\scrO_x$ to a Weyl chamber, whereby the $K$-action on
$\p$ gets replaced by the $W$-action on~$\a$. We now recall this
theorem.

\begin{lab}\label{mompolytop}%
Let $T\subset\a$ be the cone that is dual to $C$ (with respect to
the $W$-invariant inner product). If
$\Gamma=\{\beta_1,\dots,\beta_n\}$ then $T=\{x\in\a\colon
\mu_1(x)\ge0,\dots,\mu_n(x)\ge0\}$ where $\mu_1,\dots,\mu_n\in\a^\du$
is the dual basis of $\Gamma$ as in \ref{dualbas}.
For $x\in\a$, the convex hull $P_x:=\conv(Wx)$ of the (restricted)
Weyl group orbit of $x$ will play a central role. In Hamiltonian
geometry, $P_x$ is called the \emph{momentum polytope} associated to
$x$ \cite{or}, a term that we will adopt.
According to Kostant, $P_x$ is characterized as follows:
\end{lab}

\begin{prop}\label{kost33}%
\cite[Lemma 3.3]{ko}
Let $x\in C$ and $y\in\a$.
\begin{itemize}
\item[(a)]
$y\in P_x$ if and only if $x-wy\in T$ for every $w\in W$.
\item[(b)]
If $y\in C$ then $y\in P_x$ if and only if $x-y\in T$.
\end{itemize}
In particular, $P_x\cap C=\{y\in C\colon\mu_j(y)\le\mu_j(x)$ for
$j=1,\dots,n\}$.
\end{prop}

Recall that the Killing form of $\g$ restricts to a euclidean inner
product on $\p$. Let $\pi\colon\p\to\a$ denote the orthogonal
projection from $\p$ to~$\a$. Kostant's convexity theorem says:

\begin{thm}\label{kostant}%
\cite[Theorem 8.2]{ko}
If $x\in\a$ then $\pi(Kx)=P_x$.
\end{thm}

We record a few immediate consequences. Recall that
$\scrO_x=\conv(Kx)$ denotes the convex hull of the $K$-orbit of $x$
in~$\p$.

\begin{cor}\label{cor2kost}%
If $x\in\a$ then $\scrO_x\cap\a=P_x$. Hence $\pi(\scrO)=\scrO\cap\a$
holds for every $K$-invariant convex subset $\scrO$ of\/~$\p$.
\end{cor}

\begin{proof}
Theorem \ref{kostant} implies $\pi(\scrO_x)=P_x$,
and hence $\scrO_x\cap\a\subset P_x$. The reverse inclusion is
obvious since $Wx\subset Kx$.
The second assertion follows from the first, since every $K$-orbit
in $\p$ meets $\a$.
\end{proof}

\begin{cor}\label{cor3kost}%
If $x,\,y\in\a$ are $K$-conjugate then they are $W$-conjugate. Every
$K$-orbit in $\p$ intersects $C$ in a unique element.
\end{cor}

\begin{proof}
Both statements are equivalent, so it suffices to prove the second.
Let $x,\,y\in C$ be $K$-conjugate. Then $P_x=P_y$ by \ref{kostant},
so \ref{kost33}(b) implies $\pm(x-y)\in T$, whence $x-y=0$.
\end{proof}

\begin{example}\label{exsymsh}%
Let $n\ge2$ and $\g=sl(n,\R)$, the real $n\times n$ matrices of trace
zero. The resulting polar representation is the natural action of the
special orthogonal group $K=SO(n)$ on $\p=sym_0(n,\R)$, the trace
zero symmetric real matrices. The standard choice for $\a$ is to take
all diagonal matrices in $\p$. The (restricted) Weyl group is
$W=S_n$, the symmetric group, acting by permutation of the diagonal
elements of $x\in\a$. So the momentum polytope $P_x$ is the
\emph{permutahedron} of~$x$, namely the convex hull of all
permutations of $x$. Kostant's theorem \ref{kostant} specializes to
the (symmetric) Schur-Horn theorem (see \cite{hj} 4.3.45 and 4.3.48,
\cite{lrt}, \cite{sss}).
Likewise, the hermitian version of the Schur-Horn theorem arises from
\ref{kostant} if we take $\k=su(n,\C)$ and $\g=\k_\C$, so we get the
adjoint action of the special unitary group $K=SU(n)$ on (traceless)
hermitian $n\times n$ matrices, with the analogous theorem.

In Section \ref{sect:exs} we will discuss examples of polar orbitopes
in a systematic way.
\end{example}


\section{Polar orbitopes as spectrahedra}

\begin{lab}\label{cptorbtps}%
In general, orbitopes under compact connected linear groups $K$ can't
be expected to be spectrahedra. Examples are suitable
$SO(2)$-orbitopes like the $4$-dimensional Barvinok-Novik orbitope
\cite{si},
or the Grassmann orbitope $G_{3,6}$ of dimension~$20$
\cite[Theorem 7.6]{sss}, where the group $K=SO(6)$ is even
semisimple.
Using results from \cite{sch:hn} it is easy to construct orbitopes
under the $2$-torus $K=SO(2)\times SO(2)$ which are not even linear
projections of spectrahedra, for example the convex hull of
$\{(s,s^2,s^3,t,t^2,st,st^{-1})\colon s,t\in\C$, $|s|=|t|=1\}$ in
$\C^7$.
In fact it can be shown that in sufficiently high dimension, ``most''
(in a suitable sense) orbitopes under $SO(2)\times SO(2)$ fail to be
projected spectrahedra \cite{kob}.
\end{lab}

Given this general situation, we think the following theorem all
the more remarkable:

\begin{thm}\label{connpolorbsp}%
Any orbitope in a polar representation $V$ of a connected Lie group
$K$ is a spectrahedron.
\end{thm}

\begin{lab}
By Proposition \ref{dadokprop6} we may assume that $\g=\k\oplus\p$ is
a Cartan decomposition of a real semisimple Lie algebra, and that $K$
has Lie algebra $\k$ and acts on $V=\p$ via the adjoint
representation. We provide an explicit linear matrix inequality
description for any such orbitope.

Fix $\a$, $\h$ together with compatible orderings
(\ref{comporderings}), and let otherwise notation be as in
Section~\ref{lienot}.
In particular, $\Pi=\{\alpha_1,\dots,\alpha_l\}$ is the system of
simple roots of $(\g_\C,\h_\C)$, and
$\Gamma=\{\beta_1,\dots,\beta_n\}$ is the system of simple restricted
roots of $(\g,\a)$. For each $j=1,\dots,n$ choose an index
$i=i(j)\in\{1,\dots,l\}$ with $r(\alpha_{i(j)})=\beta_j$, and let
$\rho^j\colon\g_\C\to\End(V_j)$ be the (complex) irreducible
representation of $\g_\C$ with highest weight $\omega_{i(j)}$. So
$\rho^j$ is the $i(j)$-th fundamental representation of $\g_\C$, see
\ref{dualbas}. There exists an
hermitian inner product on $V_j$ making $\rho^j(x)$ self-adjoint for
all $x\in\p$ (Lemma \ref{ewprelmaxev}(a) below). In particular,
$\rho^j(x)$ has real eigenvalues for every $x\in\p$. A more precise
version of Theorem \ref{connpolorbsp} is:
\end{lab}

\begin{thm}\label{polorbspekt}%
Given $x\in\p$, the orbitope $\scrO_x=\conv(Kx)$ consists of all
$y\in\p$ such that for each $j=1,\dots,n$, all eigenvalues of
$\rho^j(y)$ are less or equal than the largest eigenvalue of
$\rho^j(x)$. In other words,
$$\scrO_x\>=\>\bigl\{y\in\p\colon\rho^j(y)\preceq c_j\cdot\id,\
j=1\dots,n\bigr\}$$
where $c_j$ is the maximal eigenvalue of $\rho^j(x)$.
\end{thm}

Upon choosing orthogonal bases of the representation spaces $V_j$,
this is an explicit description of $\scrO_x$ by linear matrix
inequalities, involving hermitian matrices in general. We remark that
the fundamental representations $\rho^j$ of $\g_\C$ are very well
known  and understood \cite{ti}, in particular so for the classical
Lie algebras.

\begin{rem}
Theorem \ref{polorbspekt} implies in particular that all faces in a
polar orbitope are exposed, since this is true in every
spectrahedron. This fact was proved before by Biliotti, Ghigi and
Heinzner \cite{bgh1}. Some spectrahedral representations contained
in, or closely related to, Theorem \ref{polorbspekt} were constructed
by Sanyal, Sottile and Sturmfels \cite{sss}, namely for symmetric
Schur-Horn orbitopes (see \ref{exsymsh}),
and also for skew-symmetric Schur-Horn orbitopes (see
\ref{unitskewherm}) and Fan orbitopes (see \ref{sonn} and
\ref{ovsso}). These latter orbitopes, as
considered in \cite{sss}, do not directly fall under the assumptions
of \ref{polorbspekt}, since the groups acting there are not connected
(full instead of special orthogonal groups). It is not hard, however,
to recover the results from \cite{sss} in our setup, see Remark
\ref{ovsso}.
In Saunderson-Parrilo-Willsky \cite{spw}, a spectrahedral
representation for the convex hull of the special orthogonal group
$SO(n)$ (and for its dual convex body) was found, see also Remarks
\ref{sonrem} and \ref{spwrem} below. Otherwise we believe that our
result is new.
\end{rem}

Before we give the proof of Theorem \ref{polorbspekt}, recall the
following well-known facts.

\begin{lem}\label{ewprelmaxev}%
Let $\rho\colon\g_\C\to\End(V)$ be a (complex) representation of
$\g_\C$.
\begin{itemize}
\item[(a)]
There exists an hermitian inner product on $V$ that makes
$\rho(x)$ self-adjoint for every $x\in\p$.
\item[(b)]
If $\rho$ is irreducible with highest weight $\omega$, and if
$x\in C$, then $\omega(x)$ is the largest eigenvalue of $\rho(x)$.
(Recall that $C$ denotes the Weyl chamber, see \ref{weyl}.)
\end{itemize}
\end{lem}

\begin{proof}
(a)
Since $\g_1:=\k\oplus i\p\subset\g_\C$ is a compact real form of
$\g_\C$,
there is an hermitian inner product on $V$ that is invariant under
this Lie algebra, i.e.\ $\rho(y)$ is anti-self adjoint for every
$y\in\g_1$. In particular, $\rho(x)$ is self-adjoint for every
$x\in\p$.

(b)
Let $\chi_1,\dots,\chi_r\in\h^\du$ be the weights of $\rho$, with
$\chi_1=\omega$. For $x\in\a$, the eigenvalues of $\rho(x)$
are $\chi_1(x),\dots,\chi_r(x)$.
Every $\chi_i$ has the form $\omega-\sum_{i=1}^lk_i\alpha_i$ with
integer coefficients $k_i\ge0$. Since $x\in C$ we have
$\alpha_i(x)\ge0$ for each index~$i$, from which the claim is
obvious.
\end{proof}

\begin{lab}
\textsc{Proof} of Theorem \ref{polorbspekt}.
Let $x\in\p$, let $c_j=\omega_{i(j)}(x)$ be the largest eigenvalue of
$\rho^j(x)$, and write $O(x):=\{y\in\p$: $\rho^j(y)\preceq
c_j\cdot\id$ for $j=1,\dots,n\}$. Both sets $\scrO_x$ and $O(x)$ are
$K$-invariant. To prove equality $\scrO_x=O(x)$, it therefore
suffices to show $\scrO_x\cap C=O(x)\cap C$, since every $K$-orbit
meets the Weyl chamber $C$.

So let $x,\,y\in C$. See \ref{dualbas} to \ref{restrfundgew} for
notation in the following discussion. Since $\rho^j$ has highest
weight $\omega_{i(j)}$, the largest eigenvalue of $\rho^j(y)$ is
$\omega_{i(j)}(y)$ (Lemma \ref{ewprelmaxev}(b)). Hence $y\in O(x)$ if
and only if $\omega_{i(j)}(y)\le\omega_{i(j)}(x)=c_j$ for
$j=1,\dots,n$. By Lemma \ref{restrfundgew}, the restriction
$r(\omega_{i(j)})\in\a^\du$ is a positive scalar multiple of $\mu_j$
(recall that $\omega_{i(j)}=\frac12|\alpha_{i(j)}|^2\lambda_{i(j)}$,
$j=1,\dots,n$).
So $y\in O(x)$ if and only if $\mu_j(y)\le\mu_j(x)$ for
$j=1,\dots,n$. By Proposition \ref{kost33}(b) this is equivalent to
$y\in P_x$. On the other hand, $y\in P_x$ is equivalent to
$y\in\scrO_x$ by Corollary \ref{cor2kost}.
\qed
\end{lab}

\begin{example}\label{sonn}%
We illustrate the statement of Theorem \ref{polorbspekt}. Let
$n\ge3$,
and consider the action of $(g,h)\in K=SO(n)\times SO(n)$ on
$x\in M_n(\R)$ by $gxh^t$. This is a polar representation of $K$
that arises from the split real form of $D_n$, i.e.\ from the simple
Lie algebra
\begin{equation}\label{natrep}%
\g\>=\>so(n,n)\>=\>\bigl\{x\in M_{2n}(\R)\colon jx+x^tj=0\bigr\},
\quad j=\begin{pmatrix}I_n&0\\0&-I_n\end{pmatrix}.
\end{equation}
Note that $\g$ consists of all block matrices
\begin{equation}\label{xuvw}%
x\>=\>\begin{pmatrix}u&w\\w^t&v\end{pmatrix}
\end{equation}
with $u,\,v,\,w\in M_n(\R)$ and $u,\,v$ skew-symmetric, and
$\p\subset\g$ is the subspace of all symmetric such matrices, i.e.\
with $u=v=0$.
As maximal commutative subspace of $\p$ we take the space $\a$ of
all matrices \eqref{xuvw} with $u=v=0$ and $w=\diag(x_1,\dots,x_n)$
diagonal. Denote such a matrix by $x=(x_1,\dots,x_n)$. The simple
roots $\beta_i=\alpha_i$ act on $x$ as $\alpha_i(x)=x_i-x_{i+1}$
($1\le i<n$) and $\alpha_n(x)=x_{n-1}+x_n$. Hence the Weyl chamber
$C$ consists of all $x\in\a$ with $x_1\ge\cdots\ge x_{n-1}\ge|x_n|$.
The fundamental weights $\mu_i=\lambda_i$ are $\lambda_i(x)=
x_1+\cdots+x_i$ ($i\le n-2$) and
$$\lambda_{n-1}(x)\>=\>\frac12\bigl(x_1+\cdots+x_{n-1}-x_n\bigr),
\quad\lambda_n(x)\>=\>\frac12\bigl(x_1+\cdots+x_{n-1}+x_n\bigr).$$
By Lemma \ref{kost33}, the momentum polytope for $x\in C$ is
described by $P_x\cap C=\{y\in C$: $\lambda_i(y)\le\lambda_i(x)$,
$i=1,\dots,n\}$.
The first fundamental representation $\rho_1$ of $\g$ is the
natural representation \eqref{natrep}, the higher ones are the
exterior powers $\rho_i=\E^i\rho_1$ ($1\le i\le n-2$). Moreover,
$\rho_{n-1}$ and $\rho_n$ are the two half-spin representations.
So $\dim(\rho_i)=\choose{2n}i$ for $i\le n-2$, and
$\dim(\rho_{n-1})=\dim(\rho_n)=2^{n-1}$. Expressing the $\rho_i$ by
matrices one arrives at explicit spectrahedral representations of
the $K$-orbitopes $\scrO_x$, for $x\in M_n(\R)$. These
representations are closely related to \cite[Theorem 4.7]{sss}, where
the group acting is $O(n)\times O(n)$ instead of our~$K$.
\end{example}

\begin{rem}\label{sonrem}%
For general $x\in\p$, none of the $n$ linear matrix inequalities
describing $\scrO_x$ in Theorem \ref{polorbspekt} can be left out.
For special $x$ this may be different. We illustrate this remark
with just one example, deferring a detailed discussion to a later
occasion.

Consider again the action of $K=SO(n)\times SO(n)$ on $M_n(\R)$, as
in \ref{sonn}, and take $x=I_n\in M_n(\R)$, the identity matrix, so
$x=(1,\dots,1)$ in notation of \ref{sonn}. The orbitope is
$\scrO_x=\conv SO(n)$, the convex hull of the group $SO(n)$. Due to
the special choice of $x$, the description of the momentum polytope
simplifies.
For $y\in C$, the condition $y_1\le1$ implies $\sum_{i=1}^ky_i\le k$
for every $k=1,\dots,n$.
So $P_x\cap C$ is already described by the two inequalities
$\lambda_1(y)\le1$ and $\lambda_{n-1}(y)\le\frac{n-2}2$.
We conclude that $\scrO_x=\conv SO(n)$ satisfies
$$\conv SO(n)\ =\ \Bigl\{y\in M_n(\R)\colon
\begin{pmatrix}0&y\\y^t&0\end{pmatrix}\preceq I,\
\rho_{n-1}\begin{pmatrix}0&y\\y^t&0\end{pmatrix}\preceq
\frac{n-2}2I\Bigr\}$$
since both sets agree when intersected with $C$. This recovers one
of the main results of Saunderson, Parrilo and Willsky
\cite[Theorem 1.3]{spw}. (In the notation of \emph{loc.\,cit.},
given a matrix $y=(y_{ij})\in M_n(\R)$, the $2^{n-1}\times 2^{n-1}$
matrix $\sum_{i,j=1}^ny_{ij}A^{(ij)}$ constructed there corresponds
to the endomorphism $\rho_n\begin{pmatrix}0&2y\\2y^t&0
\end{pmatrix}$. The extra factor $2$ accounts for the apparent
difference between their result and ours.) See \ref{spwrem} below for
a spectrahedral representation of the polar convex set $SO(n)^o$.
\end{rem}


\section{Face correspondence}

\begin{lab}\label{facecorrdfn}%
As before let $\g=\k\oplus\p$ be a Cartan decomposition of a
semisimple real Lie algebra $\g$. For general setup and notation see
Section \ref{lienot}. We continue to denote the orthogonal projection
$\p\to\a$ by~$\pi$. Let $x\in\p$, and let $P_x$ be the momentum
polytope of~$x$ (\ref{mompolytop}). If $Q$ is any face of $P_x$, then
$F_Q:=\scrO_x\cap\pi^{-1}(Q)$ is a face of $\scrO_x$.
For any $w\in W$ there exists $g\in N_K(\a)$ with $w=gZ_K(\a)$. The
projection $\pi\colon\p\to\a$ is easily seen to commute with the
action of $N_K(\a)$, and therefore $F_{wQ}=gF_Q$ holds.
Hence the assignment $Q\mapsto F_Q$ induces a map from $W$-orbits of
faces of $P_x$ to $K$-orbits of faces of $\scrO_x$.

The following theorem asserts, in particular, that this map is
bijective. This fact was originally proved by Biliotti, Ghigi and
Heinzner \cite[Theorem~1.1]{bgh1}. We give a new proof that we
think is considerably easier. Note however that \cite{bgh1} proves a
more precise result, implying in particular that the faces $F_Q$ of
$\scrO_x$ are themselves orbitopes under suitable groups.
\end{lab}

\begin{thm}\label{orbitcorr}%
Let $x\in\p$, let $F$ be a face of the orbitope $\scrO_x$.
\begin{itemize}
\item[(a)]
There exists a face $Q$ of $P_x$ and an element $g\in K$ such that
$F=gF_Q$.
\item[(b)]
If $Q'$ is another face of $P_x$ with $F\subset g'F_{Q'}$ for some
$g'\in K$, then there exists $w\in W$ such that $wQ\subset Q'$.
\end{itemize}
In particular, $Q\mapsto F_Q$ induces a bijective correspondence
between $W$-orbits of faces of $P_x$ and $K$-orbits of faces of
$\scrO_x$, compatible with inclusion of faces.
\end{thm}

For the proof observe the following lemma:

\begin{lem}\label{keylem}%
Let $x\in\p$, let $Q$ be a face of $P_x$, and let $y\in P_x$ with
$P_y\cap Q\ne\emptyset$. Then $wy\in Q$ for some $w\in W$.
\end{lem}

\begin{proof}
We can assume $Q\ne P_x$,
so there is a supporting hyperplane $H\subset\a$ of $P_x$ with
$Q=H\cap P_x$. Since $P_y\subset P_x$ and $H\cap P_y$ is not empty,
the hyperplane $H$ is a supporting hyperplane of $P_y$ as well.
In particular, $Q'=H\cap P_y$ is a face of $P_y$, and therefore
contains an extreme point $y'$ of $P_y$.
Thus $y'\in Q'\subset H\cap P_x=Q$, and $y'=wy$ for some $w\in W$
since $P_y=\conv(Wy)$.
\end{proof}

\begin{lab}
\textsc{Proof} of Theorem \ref{orbitcorr}.
\smallskip

(a)
Let $F$ be a face of $\scrO_x$. By Theorem \ref{connpolorbsp},
$\scrO_x$ is a spectrahedron, so all faces are exposed. Hence there
exist $z\in\p$ and $c\in\R$ such that $H=\{y\in\p\colon\bil yz=c\}$
is a supporting hyperplane of $\scrO_x$ with $H\cap\scrO_x=F$. Upon
replacing $F$ with $gF$ for some $g\in K$ we can assume $z\in\a$,
since $z$ is $K$-conjugate to an element of $\a$.
Then $H\cap\a$ is a supporting hyperplane of $P_x$, and so
$Q:=H\cap P_x$ is a face of $P_x$. Clearly $F=F_Q$.
\smallskip

(b)
By (a) it suffices to show: If $Q,\,Q'$ are faces of $P_x$, and if
$gF_Q\subset F_{Q'}$ for some $g\in K$, then there exists
$w\in W$ with $wQ\subset Q'$.
Let $y\in\relint(Q)$. Since $Q\subset F_Q$ we have $gy\in F_{Q'}$,
and therefore $\pi(gy)\in Q'$.
On the other hand, $\pi(gy)\in\pi(Ky)=P_y$. So Lemma \ref{keylem}
applies and shows $wy\in Q'$ for some $w\in W$.
Since $y\in\relint(Q)$, this implies $wQ\subset Q'$.

In particular, if $Q,\,Q'$ are faces of $P_x$ for which $F_Q$ and
$F_{Q'}$ are $K$-conjugate, then $Q$ and $Q'$ are $W$-conjugate.
\end{lab}

Recall that a face $Q$ of a polytope $P$ is called a \emph{facet} if
$\dim(Q)=\dim(P)-1$.

\begin{cor}\label{maxpropfac}%
The maximal proper faces of $\scrO_x$ are precisely the
$K$-conjugates of the faces $F_Q$, where $Q$ is a facet of~$P_x$.
\qed
\end{cor}

\begin{lab}\label{polarset}%
We apply this result to the study of the coorbitope $\scrO_x^o$.
First recall the definition of the polar of a convex set. Let $V$
be a real vector space, $\dim(V)<\infty$.
For any set $M\subset V$ let $M^o=\{l\in V^\du\colon\all x\in M$
$l(x)\le1\}$, the \emph{polar set} of~$M$. Usually a euclidean
inner product $\bil\hvar\hvar$ on $V$ will be fixed, then we identify
$M^o$ with the set $\{y\in V\colon\all x\in M$ $\bil xy\le1\}$.
If $M$ is compact and $0$ is an interior point of $M$, the same holds
for~$M^o$.

If the compact group $K$ acts on $V$ and $x\in V$, the polar set of
the orbitope $\scrO_x=\conv(Kx)$ is called the \emph{coorbitope}
of~$x$ \cite{sss}.
Clearly the group $K$ acts on $\scrO_x^o$, but in general
$\scrO_x^o$ won't be a $K$-orbitope. Below (\ref{facetsconj} and
\ref{biorbitop}) we'll identify those cases when this happens.
\end{lab}

\begin{lab}\label{fulldim}%
For any irreducible abstract root system $(V,\Sigma)$ (possibly
non-reduced), the Weyl group $W$ acts irreducibly on~$V$.
For any $0\ne x\in V$, the polytope $P_x=\conv(Wx)$ therefore
contains an open neighborhood of the origin. If $(V,\Sigma)$ is not
necessarily irreducible and $x\in V$, it follows that the polytope
$P_x$ is full-dimensional if and only if every irreducible component
of $\Sigma$ contains a root $\alpha$ with $\alpha(x)\ne0$. Moreover
in this case, $0$ is an interior point of $P_x$.

Let $K\to SO(V)$ be a polar representation, and let $x\in V$. The
previous discussion implies that when $V$ is irreducible, the origin
is an interior point of $\scrO_x$ as soon as $x\ne0$.
When $V$ is an arbitrary polar representation of $K$, let
$V=\bigoplus V_i$ be the decomposition into irreducible $K$-modules
as in \ref{dirprod}, and let $x=\sum_ix_i\in V$ with $x_i\in V_i$.
Then $\scrO_x=\scrO_{x_1}\times\cdots\times\scrO_{x_n}$ by
\ref{dirprod}. Therefore $\scrO_x$ is full-dimensional in $V$ iff $0$
is an interior point of $\scrO_x$, and both are equivalent to
$x_i\ne0$ for each index~$i$. It is also equivalent that
$\pi(Kx)=\scrO_x\cap\a$ (\ref{cor2kost}) is full-dimensional in $\a$.
\end{lab}

When studying the orbitope $\scrO_x$, we can obviously assume that
$\scrO_x$ is full-dimen\-sional (or equivalently, $0$ is an interior
point of $\scrO_x$), by the previous discussion.

\begin{prop}\label{oxoorbits}%
Let $x\in\p$ such that $\scrO_x$ is full-dimensional, and let
$\scrO_x^o\subset\p$ be the associated coorbitope. The $K$-orbits of
extreme points of $\scrO_x^o$ are in natural bijective correspondence
with the $W$-orbits of facets of the polytope $P_x$. In particular,
$\scrO_x^o$ is the convex hull of finitely many $K$-orbits in~$\p$.
\end{prop}

\begin{proof}
First recall the following general and easy fact (see
\cite[2.1.4]{sn}, for example). Let $\scrO\subset\R^n$ be any compact
convex body which contains a neighborhood of~$0$, and let
$\scrO^o\subset\R^n$ be the convex body polar to $\scrO$. For any
face $F$ of $\scrO$ let $\wh F=\{y\in\scrO^o\colon\all x\in F$
$\bil xy=1\}$. Then $\wh F$ is an exposed face of $\scrO^o$, and
$F\mapsto\wh F$ restricts to an inclusion-reversing bijection between
exposed faces of $\scrO$ and exposed faces of $\scrO^o$, with inverse
map $G\mapsto\wh G$.

To prove the proposition we can assume $x\in C$. Let $Fac(x)$ be a
set of representatives of the $W$-orbits of facets of $P_x$. Let $z$
be an exposed extreme point of $\scrO_x^o=\{y\in\p\colon\all g\in K$
$\bil{gx}y\le1\}$, and write $G_z:=\wh z=\{y\in\scrO_x\colon
\bil yz=1\}$. By the fact just recalled, $G_z$ is a maximal face of
$\scrO_x$,
and so $G_z=gF_Q$ for some $Q\in Fac(x)$ and some $g\in K$
(\ref{orbitcorr}, \ref{maxpropfac}). It is easily checked that
$G_{hz}=hG_z$ for any $h\in K$.
If $u\in\scrO_x^o$ is another exposed extreme point, and if
$G_u=hF_Q$ for some $h\in K$, then $G_{gh^{-1}u}=gh^{-1}G_u=
G_z$, whence $gh^{-1}u=z$, so $u$ and $z$ are $K$-conjugate. This
shows that the exposed extreme points of $\scrO_x^o$ consist of
finitely many $K$-orbits, each of them corresponding to a different
$W$-orbit of facets of $P_x$. Since exposed extreme points are dense
within all extreme points (Straszewicz' theorem, e.g.\ \cite{sn}
1.4.7), we conclude that all extreme points of $\scrO_x^o$ are
exposed.

For each facet $Q$ of $P_x$ we claim conversely that $\wh{F_Q}$ is an
(exposed) extreme point of $\scrO_x^o$. Indeed, otherwise
$\wh{F_Q}$ would be a minimal exposed face of $\scrO_x^o$ of
dimension $\ge1$. But such a face cannot exist, since all extreme
points of $\scrO_x^o$ are exposed.
Altogether we have proved the bijection between $K$-orbits of
extreme points of $\scrO_x^o$ and $W$-orbits of facets of $P_x$.
\end{proof}

\begin{cor}\label{facetsconj}%
Let $x\in\p$. The coorbitope $\scrO_x^o$ is a $K$-orbitope itself
if, and only if, all facets of the polytope $P_x$ are $W$-conjugate.
\qed
\end{cor}

We will determine these cases explicitly in the next section, after
having discussed the facets of $P_x$ in more detail.

Note that under the equivalent conditions of \ref{facetsconj},  the
coorbitope $\scrO_x^o$ is a spectrahedron itself, by Theorem
\ref{connpolorbsp}.
In Sect.~\ref{sec:doubly} we will uncover many more cases where this
holds.


\section{Facets of the momentum polytope}

In the previous section, a close relation was established between the
faces of the orbitope $\scrO_x$ and the faces of the momentum
polytope $P_x$. The faces of the latter can be described in terms of
root data. We start by recalling this description.

\begin{lab}
Let $(V,\Sigma)$ be an abstract root system (which may be
non-reduced), fix an ordering $\le$ on $V$, and let
$\Gamma=\{\beta_1,\dots,\beta_n\}$ be the corresponding system of
simple positive roots. Let $\mu_1,\dots,\mu_n$ be the dual basis of
$\Gamma$ in $V$, so $\bil{\beta_i}{\mu_j}=\delta_{ij}$ for
$1\le i,\,j\le n$. Let $W$ be the Weyl group, and let
$C=\{x\in V\colon\bil{\beta_i}x\ge0$, $1\le i\le n\}$, the closed
Weyl chamber associated to $\Gamma$.

Let $x\in C$ be a given point and write $P_x=\conv(Wx)$.
A subset $I\subset\Gamma$ is said to be \emph{$x$-connected} if every
connected component of $I$ contains a root $\beta$ with
$\bil\beta x\ne0$. (Of course, connectedness notions refer to the
Dynkin graph.) Let $W_I$ be the subgroup of $W$ generated by the
root reflections $s_\beta$ where $\beta\in I$. The following result
is quoted from Casselman \cite[Theorem 3.1]{ca}, where it is proved
in the more general context of arbitrary finite Coxeter groups. As
Casselman remarks, the result is already implicit in much older
work of Satake \cite{sat} and Borel-Tits \cite{boti}.
A related discussion can also be found in \cite[\S6]{bgh2}
and \cite[\S4]{bgh1}.
\end{lab}

\begin{thm}\label{facmompol}%
Let $x\in C$.
The map $I\mapsto\conv(W_Ix)$ induces a bijection between the
$x$-connected subsets $I$ of\/ $\Gamma$ and the $W$-orbits of
faces of $P_x=\conv(Wx)$. For any such $I$ one has
$\dim\conv(W_Ix)=|I|$.
\end{thm}

Here we are mainly interested in the facets of $P_x$. Assume that
$P_x$ is full-dimensional in $V$, or equivalently, that every
connected component of $\Gamma$ contains a root $\beta$ with
$\bil\beta x\ne0$ (see \ref{fulldim}).
For facets the theorem gives:

\begin{cor}\label{facetsmompol}%
Let $x\in C$ such that $P_x$ is full-dimensional, and let $I(x)$
denote the set of indices $i\in\{1,\dots,n\}$ for which
$\Pi\setminus\{\beta_i\}$ is $x$-connected. For $i\in I(x)$ let
$$P_x(i)\>:=\>\{y\in P_x\colon\mu_i(y)=\mu_i(x)\}.$$
Then $P_x(i)$ is a facet of $P_x$. Conversely, every facet of
$P_x$ is $W$-conjugate to $P_x(i)$ for a unique index $i\in I(x)$.
\end{cor}

\begin{proof}
Let $i\in I(x)$. Clearly, $P_x(i)$ is a face of $P_x$, and is proper
since $P_x$ is full-dimensional.
Let $W':=W_{\Pi\setminus\beta_i}=\langle s_\beta\colon\beta\in\Pi$,
$\beta\ne\beta_i\rangle$. Then $\conv(W'x)\subset P_x(i)$ holds
since $\mu_i(s_{\beta_j}(y))=\mu_i(y)$ for every $j\ne i$.
By Theorem \ref{facmompol}, $\conv(W'x)$ is a facet of $P_x$, so we
have equality. The remaining assertion follows directly from Theorem
\ref{facmompol} as well.
\end{proof}

We will also use the following (well-known) fact:

\begin{lem}\label{stripos}%
Let $x\in C$ such that $P_x$ is full-dimensional. Then $\mu_i(x)>0$
for every $i=1,\dots,n$.
\end{lem}

\begin{proof}
It is enough to prove this in the case where the root system is
simple and $0\ne x\in C$. Since $\beta_j(x)\ge0$ for all
$j=1,\dots,n$ and $\beta_j(x)>0$ for at least one~$j$, the lemma
follows from the fact that the inverse of the Cartan matrix has
strictly positive coefficients \cite{luti}.
\end{proof}

\begin{lab}\label{polarpoints}%
Now again consider the adjoint representation of $K$ on~$\p$. We
apply Corollary \ref{facetsmompol} to the system $\Sigma$ of
restricted roots of $(\g,\a)$. In this way we are going to identify
explicitly the $K$-orbits of extreme points of the coorbitope
$\scrO_x^o$ (see Proposition \ref{oxoorbits}).

Let $x\in C\subset\a$ such that $P_x$ is full-dimensional, and let
$I(x)=\{i\in\{1,\dots,n\}\colon\Gamma\setminus\{\beta_i\}$ is
$x$-connected$\}$ (as in \ref{facetsmompol}). For every $i\in I(x)$
we have the facet $P_x(i)=\{y\in P_x\colon\mu_i(y)=\mu_i(x)\}$ of
$P_x$. For easier notation, let us
write $F_i$ instead of $F_{P_x(i)}=\{y\in\scrO_x\colon\pi(y)\in
P_x(i)\}$. By Theorem \ref{orbitcorr}, $F_i$ is a maximal proper face
of $\scrO_x$, and every maximal proper face is $K$-conjugate to $F_i$
for a unique index $i\in I(x)$ (using also \ref{facetsmompol}). Given
$i\in I(x)$, there is a unique extreme point $z_i$ of $\scrO_x^o$
that corresponds to $F_i$ under polarity, characterized by
$F_i=\wh{z_i}=\{y\in\scrO_x\colon\bil y{z_i}=1\}$ (see proof of
\ref{oxoorbits}). The points $z_i$, for $i\in I(x)$,
represent the pairwise different $K$-orbits of extreme points in
$\scrO_x^o$. Using \ref{facetsmompol} we identify these points as
follows.

Let $i\in I(x)$, let $h_{\mu_i}\in\a$ be the element satisfying
$\bil{h_{\mu_i}}y=\mu_i(y)$ for all $y\in\a$. Note that $C$ is the
cone generated by $h_{\mu_1},\dots,h_{\mu_n}$. By Lemma \ref{stripos}
we have $\mu_i(x)>0$. We claim that $z_i=h_{\mu_i}/\mu_i(x)$.

Indeed, the element $z:=h_{\mu_i}/\mu_i(x)$ satisfies
$\bil yz=\bil{\pi(y)}z=\mu_i(\pi(y))/\mu_i(x)$ for all $y\in\scrO_x$.
Since $\pi(y)\in P_x$ (\ref{cor2kost}), this shows $\bil yz\le1$,
with equality if and only if $\pi(y)\in P_x(i)$.
So we have proved:
\end{lab}

\begin{cor}\label{orbitsoxo}%
Let $x\in C$ such that $P_x$ is full-dimensional. The coorbitope
$\scrO_x^o\subset\p$ is the convex hull of the union of the
$K$-orbits of the elements $z_i=h_{\mu_i}/\mu_i(x)\in C$, for $i$
running through $I(x)$.
\qed
\end{cor}

A particularly interesting case arises when $|I(x)|=1$, i.e.\ the
polytope $P_x$ has only one $W$-orbit of facets. By \ref{facetsconj}
it is equivalent that the coorbitope $\scrO_x^o$ is a $K$-orbitope
itself. We will say that $\scrO_x$ is a ($K$-) \emph{biorbitope} in
this case, and we can characterize it as follows:

\begin{thm}\label{biorbitop}%
Let $0\ne x\in C$ such that $P_x$ is full-dimensional. Then $\scrO_x$
is a $K$-biorbitope if, and only if, the Lie algebra $\g$ is simple,
the restricted root system $\Gamma$ is not of type $D_n$ ($n\ge4$) or
$E_n$ ($n=6,7,8$), and the following holds: There is a simple
restricted root $\beta\in\Gamma$ with $\gamma(x)=0$ for all
$\gamma\in\Gamma\setminus\{\beta\}$, and such that
$\Gamma\setminus\{\beta\}$ is connected.
\end{thm}

In other words, the condition is saying that $\beta(x)\ne0$ for only
one simple restricted root $\beta$, that $\beta$ sits at an end of
the restricted Dynkin graph $\Gamma$, and that $\Gamma$ has at most
one other end.
In \ref{biorb1} and \ref{biorb2} we'll make all biorbitopes explicit
for the classical Lie algebras.

\begin{proof}
By Corollary \ref{orbitsoxo}, $\scrO_x$ is a biorbitope if and only
if $P_x$ has only one $W$-orbit of facets. One sees immediately that
this can hold only when the restricted root system $\Sigma$ is
irreducible.
Therefore we may assume that the Lie algebra $\g$ is simple.

Let $\Gamma_1=\{\beta\in\Gamma\colon\beta(x)\ne0\}$. We say that
$\beta\in\Gamma$ is a boundary root if $\Gamma\setminus\{\beta\}$ is
connected. If $\Gamma_1$ contains a non-boundary root then $P_x$ has
two non-conjugate facets. Indeed, choose two different boundary roots
$\beta_i$ and $\beta_j$. Then $P_x(i)$ and $P_x(j)$ are both facets
of $P_x$, and are not $W$-conjugate, according to Corollary
\ref{facetsmompol}.
Exactly the same argument works if $\Gamma_1$ contains two different
boundary roots $\beta_i$, $\beta_j$.

So all facets of $P_x$ can only be $W$-conjugate if $\Gamma_1$
consists of just one single boundary root. Conversely, if this is the
case then the conjugacy classes of
facets of $P_x$ correspond precisely to the remaining boundary roots.
This proves the equivalence in the theorem, since $D_n$ ($n\ge4$) and
$E_n$ ($n=6,7,8$) are precisely the simple root systems with more than
two boundary roots.
\end{proof}

\begin{example}\label{dn2orb}%
Let $n\ge4$ and $\g=so(n,n)$. For the description of $\a$, $C$, the
$\alpha_i$ and $\lambda_i$ see \ref{sonn}. The restricted root system
is of type $D_n$. If we take $x=(1,\dots,1)$ as in \ref{sonrem}, we
have $\alpha_i(x)=0$ for all $i\ne n$, so $I(x)=\{n\}$ is a singleton
set. Yet $P_x$ has two $W$-orbits of facets, represented by the
facets $P_x(1)=\{y\in P_x\colon y_1=1\}$ and $P_x(n-1)=\{y\in P_x$:
$y_1+\cdots+y_{n-1}-y_n=n-2\}$. Hence
the orbitope $\scrO_x=\conv SO(n)$ has two $K$-orbits of maximal
dimensional faces, a fact already proved in
\cite[Theorem 4.11]{sss}. This means that the coorbitope
$\scrO_x^o=SO(n)^o$ is not an orbitope, rather
$$SO(n)^o\>=\>\conv(Kz_1\cup Kz_{n-1})$$
by Corollary \ref{orbitsoxo}, where $z_1=(1,0,\dots,0)$ and
$z_{n-2}=\frac1{n-2}(1,\dots,1,-1)$ (notation as in \ref{sonn} and
\ref{sonrem}).
A similar remark applies when $x=(1,0,\dots,0)$ (here $\scrO_x$ is
the unit ball of the nuclear norm on $M_n(\R)$, see
\ref{expqexcept})
and of  $x=(1,\dots,1,-1)$ (here $\scrO_x$ is the convex hull of
$O^\minus(n)$,
which is of course linearly isomorphic to $\conv SO(n)$).
\end{example}

Remarkably, whenever $\scrO_x$ is a biorbitope, the coorbitope
$\scrO_x^o$ is a positive scaling of $\scrO_x$:

\begin{thm}\label{oxovsox}%
Let $\g$ be simple and $0\ne x\in C$, and assume that $\scrO_x$ is a
biorbitope, i.e.\ $|I(x)|=1$. Then there is a real number $c>0$ such
that $\scrO_x^o=c\cdot\scrO_x$.
\end{thm}

\begin{proof}
We have $I(x)=\{i\}$ where $\beta_i\in\Gamma$ is a boundary root
(Theorem \ref{biorbitop}). By Corollary \ref{orbitsoxo},
$\scrO_x^o$ is the convex hull of $Kz_i$ where
$z_i=h_{\mu_i}/\mu_i(x)$. The element $x$ itself is a scalar multiple
of $h_{\mu_i}$ since $\beta_j(x)=0$ for all $\beta_j\in
\Gamma\setminus\{\beta_i\}$. More precisely $x=\beta_i(x)h_{\mu_i}$,
since both elements give the same value under every $\beta_j$.
This implies $\mu_i(x)=\beta_i(x)\mu_i(h_{\mu_i})=
\beta_i(x)\,|\mu_i|^2$. So $z_i=h_{\mu_i}/\mu_i(x)=
x/\beta_i(x)^2|\mu_i|^2$, and therefore
$$\scrO_x^o\>=\>\frac1{\beta_i(x)^2\cdot|\mu_i|^2}\,\scrO_x.$$
\end{proof}


\section{Doubly spectrahedral orbitopes}\label{sec:doubly}%

Following Saunderson, Parrilo and Willsky \cite{spw} we use the term
\emph{doubly spectrahedral convex sets} to refer to convex sets $S$
in $\R^n$ for which both $S$ and the polar convex set $S^o$ are
spectrahedra. As remarked in \cite{spw}, it is a very special
phenomenon that the polar set of a spectrahedron is again a
spectrahedron. Apart from polyhedra (which have this property for
obvious reasons) it seems that only one other distinct family of
doubly spectrahedral convex sets is known, namely the homogeneous
convex cones (Vinberg \cite{vi} and Chua \cite{ch}, see
\cite[6.1]{spw}). In addition, the convex hull of the matrix group
$SO(n)$ is doubly spectrahedral for every $n\ge1$, by the main
theorem of \cite{spw}. In fact, explicit spectrahedral
representations for both $\conv SO(n)$ and $SO(n)^o$ were constructed
in \cite{spw}.

Below we show that all polar orbitopes $\scrO_x$ with ``rational
coordinates'' are doubly spectrahedral as well. Moreover we'll give
explicit linear matrix inequality representations for those orbitopes
and their polars. As a particular case, we recover the results from
\cite{spw}, see Remark \ref{spwrem} below.

Let $\g=\k\oplus\p$ be a real semisimple Lie algebra with Cartan
decomposition, and consider the adjoint representation of $K$ on $\p$
as before. We use notation and conventions from Section \ref{lienot}.
In particular, $\a$ is a maximal abelian subspace of $\p$, and
$C\subset\a$ is the Weyl chamber with respect to the fixed ordering
on~$\a$. As before, let $\Gamma=\{\beta_1,\dots,\beta_n\}\subset
\a^\du$ be the simple positive restricted roots.

\begin{dfn}
Given $x\in\a$, we say that the $K$-orbitope $\scrO_x=\conv(Kx)$ has
\emph{rational coordinates} if there is $b\in\R$ such that
$\beta_j(x)\in\Q b$ for $j=1,\dots,n$.
\end{dfn}

Since any two choices of $\a\subset\p$ are conjugate under $K$
\cite[6.51]{kn},
and since every $K$-orbit in $\p$ intersects $\a$ in one full
$W$-orbit (Corollary \ref{cor3kost}), the property of having rational
coordinates depends only on the orbit $Kx$, and neither on the choice
of $\a$ nor on the particular choice of a representative of $Kx$
in~$\a$.

\begin{thm}\label{doublspect}%
Let $\scrO_x$ be a polar orbitope with rational coordinates. Then
both $\scrO_x$ and $\scrO_x^o$ are spectrahedra.
\end{thm}

For $\scrO_x$, a spectrahedral representation has been given in
Theorem \ref{polorbspekt}. In \ref{lmioxo} below we explain how to
find one for $\scrO_x^o$. Explicit descriptions of these orbitopes
are contained in the next section, c.f.\ Remark \ref{doublspectex}.
\smallskip

Given $x\in\a$, let $P_x=\conv(Wx)$ be the momentum polytope of $x$
as before, and let $P_x^o$ be the polar set of $P_x$ in $\a$, i.e.
$$P_x^o\>=\>\{y\in\a\colon\all w\in W\ \bil{wx}{y}\le1\}.$$
We have the following lemma:

\begin{lem}\label{lab7}%
If $x\in C$ then $P_x^o\cap C=\{y\in C\colon\bil xy\le1\}$.
\end{lem}

\begin{proof}
If $y\in C$ then $\bil{wx}y\le\bil xy$ for every $w\in W$
\cite[Lemma 3.2]{ko}.
Therefore, if $\bil xy\le1$ then $y\in P_x^o$. The opposite inclusion
is trivial from the definition.
\end{proof}

Recall that $\pi\colon\p\to\a$ denotes orthogonal projection to $\a$.
Forming the polar convex body commutes with projection to (or
intersection with)~$\a$:

\begin{lem}\label{lab3}%
Let $\scrO\subset\p$ be a $K$-invariant convex set, and let
$Q:=\scrO\cap\a=\pi(\scrO)$ (\ref{cor2kost}). Then
$\pi(\scrO^o)=\scrO^o\cap\a=Q^o$ (the polar set of $Q$ in $\a$).
\end{lem}

\begin{proof}
For $y\in\a$ and $z\in\p$ we have $\bil y{\pi(z)}=\bil yz$. From this
the lemma follows immediately.
\end{proof}

\begin{lab}
\textsc{Proof} of Theorem \ref{doublspect}.
We can assume $x\in C$. Since $x$ has rational coordinates we can
assume $\beta_j(x)\in\Q$ for $j=1,\dots,n$, after scaling $x$ with a
suitable positive real number. So there exist rational numbers
$c_j\ge0$ such that $\bil xy=\sum_{j=1}^nc_j\mu_j(y)$ for all
$y\in\a$ (namely $c_j=\beta_j(x)$).
Hence, and by Lemma \ref{restrfundgew}, there are an integer $k\ge1$
and an integral dominant weight $\omega\in\h^\du$ of $(\g_\C,\h_\C)$
such that $k\cdot\bil xy=\omega(y)$ for all $y\in\a$.
By the highest weight theorem, there is an irreducible representation
$\rho$ of $\g_\C$ with highest weight $\omega$. From Lemmas
\ref{lab7} and \ref{lab3} we get
\begin{equation}\label{polarcapc}%
\scrO_x^o\cap C\>=\>P_x^o\cap C\>=\>\{y\in C\colon\bil xy\le1\}
\end{equation}
We claim that \eqref{polarcapc} implies
\begin{equation}\label{claim}%
\scrO_x^o\>=\>\{y\in\p\colon\rho(y)\preceq k\cdot\id\}.
\end{equation}
Indeed, both sets in \eqref{claim} are $K$-invariant, so it suffices
to check that their intersections with $C$ coincide. For $y\in C$ the
largest eigenvalue of $\rho(y)$ is $\omega(y)=k\bil xy$ (Lemma
\ref{ewprelmaxev}(b)). So \eqref{claim} follows indeed from
\eqref{polarcapc}, and the theorem is proved.
\end{lab}

\begin{rem}\label{lmioxo}%
The highest weights of irreducible representations of $(\g_\C,\h_\C)$
are the nonnegative integral linear combinations of the fundamental
weights $\omega_i=\frac12|\alpha_i|^2\lambda_i$ ($i=1,\dots,l)$. The
restriction of $\omega_i\in\h^\du$ to $\a$ is $0$ if $r(\alpha_i)=0$,
and is $\frac1{2m}|\alpha_i|^2\mu_j$ if $r(\alpha_i)=\beta_j\ne0$,
where $m\in\{1,2\}$ is the number of simple roots in $\Pi$ that
restrict to $\beta$ (Lemma \ref{restrfundgew}). Since the
$|\alpha_i|^2$ are explicit rational numbers, we see how to find, for
given $x\in\a$ with rational coordinates, a real number $c>0$ and an
integral dominant weight $\omega$ of $(\g_\C,\h_\C)$ such that
$\bil{cx}y=\omega(y)$ for all $y\in\a$.
\end{rem}

\begin{example}\label{spwrem}%
We illustrate the previous remark with the example already studied in
\ref{sonn}, so consider the action of $K=SO(n)\times SO(n)$ on
$M_n(\R)$ for $n\ge3$. We take the identity matrix $x=I_n$ as in
\ref{sonrem} and are looking for a linear matrix inequality
description of the coorbitope $\scrO_x^o=(\conv SO(n))^o$. The
orbitope $\scrO_x$ has rational coordinates since $\alpha_i(x)=0$
for $1\le i<n$ and $\alpha_n(x)=2$ (see \ref{sonn}). Since
$\bil xy=\sum_{i=1}^ny_i=2\lambda_n(y)$ for $y\in\a$, the procedure
in \ref{lmioxo} leads to the spectrahedral representation
$$\bigl(\conv SO(n)\bigr)^o\>=\>\Bigl\{y\in M_n(\R)\colon\rho_n
\begin{pmatrix}0&y\\y^t&0\end{pmatrix}\preceq\frac12\,\id\Bigr\}$$
where $\rho_n$ is the $n$-th fundamental representation. This is in
accordance with Saunderson, Parrilo and Willsky
\cite[Theorem 1.1]{spw}, c.f.\ the remark in \ref{sonrem}.

For $n=3$, $\scrO_x^o$ is a $K$-orbitope itself. For $n\ge4$,
$\scrO_x^o$ is the convex hull of two $K$-orbits, but not of one
(Example \ref{dn2orb}).
\end{example}


\section{Examples}\label{sect:exs}

We describe all irreducible polar representations that arise from
semisimple Lie algebras of classical type.
Roughly, these are the well-known unitary group actions on
rectangular matrices, and on (skew-) hermitian resp.\ (skew-)
symmetric square matrices, over $\K=\R,\,\C$ or $\bbH$, where $\bbH$
is the skew-field of Hamilton quaternions. (For $\K=\bbH$ there is
no action on (skew-) symmetric matrices.)
In each case we mention a standard choice of a maximal abelian
subspace $\a$ and of a Weyl chamber $C$. Using Kostant's results, in
particular Proposition \ref{kost33} and Corollary \ref{cor2kost},
this allows us to give explicit descriptions of the respective
orbitopes in all cases. Naturally, this uses the description of the
(reduced) root systems and of the fundamental weights, for which
there are many references (e.g.\ \cite{kn}, \cite{ti}, \cite{lie3}).
We will see that the corresponding orbitopes can be described in
terms of Ky Fan norm balls, which in turn are defined using singular
values of matrices.

\begin{lab}
First recall the singular value decomposition. Let always $\K$ be one
of $\R$, $\C$ or $\bbH$,
and let $U(n,\K)$ denote the unitary group over $\K$, i.e.\
$U(n,\K)=\{g\in M_n(\K)\colon gg^*=I_n\}$ where $g^*=\ol g^t$. So
$U(n,\R)=O(n)$ is the real orthogonal group, $U(n,\C)=U(n)$ is the
usual (complex) unitary group and $U(n,\bbH)=Sp(n)$ is the symplectic
group. Given a (rectangular) matrix $x\in M_{m\times n}(\K)$ where
$m\ge n$, there exist unitary matrices $u\in U(m,\K)$ and
$v\in U(n,\K)$ such that $uxv^*=\diag(a_1,\dots,a_n)$ with real
numbers $a_1\ge\cdots\ge a_n\ge0$. The $a_i$ are uniquely determined
by $x$, they are called the \emph{singular values} of~$x$ and denoted
$\sigma_i(x):=a_i$ ($1\le i\le n$). For $\K=\R$ or $\C$ this is
classical (e.g.\ \cite[2.6]{hj}), here $a_1^2,\dots,a_n^2$ are the
eigenvalues of the psd hermitian matrix $x^*x$. For $\K=\bbH$,
essentially the same is true (with eigenvalues replaced by right
eigenvalues), but less well-known; see \cite[7.2]{zh} and
\cite[5.7]{fp} for details.
\end{lab}

\begin{lab}\label{kyfanorm}%
Let $\K=\R,\,\C$ or $\bbH$, and let $V=M_{m\times n}(\K)$ with
$m\ge n$. For $k=1,\dots,n$ and $x\in V$ let
$$\Vert x\Vert_k\>:=\>\sigma_1(x)+\cdots+\sigma_k(x),$$
sum of the $k$ largest singular values of $x$. This defines a
matrix norm on the space of matrices, the \emph{$k$-th Ky Fan norm}
(\cite{fa} and \cite{hj}, 7.4.8 and 7.4.10).
In particular, all balls with respect to any of these norms are
convex.
Note that the first Ky Fan norm $\Vert x\Vert_1=\sigma_1(x)$ is the
operator norm of $x$. The last one $\Vert x\Vert_n=\sum_{i=1}^n
\sigma_i(x)$ is called the nuclear norm and often denoted
$\Vert x\Vert_*$.
\end{lab}

In view of Remark \ref{dirprod}, we restrict our discussion of
classical polar orbitopes to orbitopes that arise from simple real
Lie algebras $\g$ of classical type.

\begin{example}\label{expqnonexcept}%
Let $m\ge n\ge1$. Consider $K=SO(m)\times SO(n)$ (case $\K=\R$)
resp.\ $K=S(U(m)\times U(n))$ (case $\K=\C$) resp.\ $K=Sp(m)\times
Sp(n)$ (case $\K=\bbH$), together with the action of $(u,v)\in K$ on
$x\in V=M_{m\times n}(\K)$ by $uxv^*$. This is a polar
representation, arising from the simple Lie algebra $\g=so(m,n)$
resp.\ $\g=su(m,n)$ resp.\ $\g=sp(m,n)$ (assume $m+n\ne2,\,4$ if
$\K=\R$).
For $\a\cong\R^n$ we can take the space of real matrices that are
diagonal in the upper $n$ rows and zero below. If $(a_1,\dots,a_n)$
is the diagonal part of such a matrix $x$, let us write $x_i:=a_i$.
The case $\K=\R$ and $\g=so(m,n)$ with $m=n$ is exceptional (see
\ref{expqexcept} below), so let us first discard it. In all other
cases the Weyl chamber $C$ consists of all $x\in\a$ with
$x_1\ge\cdots\ge x_n\ge0$. Moreover, using Lemma \ref{kost33}
we see that
$$P_x\cap C\>=\>\bigl\{y\in C\colon y_1+\cdots+y_k\le x_1+\cdots+x_k\
(k=1,\dots,n)\bigr\}$$
If $x\in C$ then clearly $x_i=\sigma_i(x)$, the $i$-th singular value
of~$x$. It follows for arbitrary $x\in V$ that
$$\scrO_x\>=\>\bigl\{y\in V\colon\Vert y\Vert_k\le\Vert x\Vert_k\
(k=1,\dots,n)\bigr\}$$
since both sets are $K$-invariant and their intersections with $C$
coincide. So $\scrO_x$ is an intersection of balls with center $0$
with respect to the Ky Fan norms $\Vert\cdot\Vert_k$ ($k=1,\dots,n$),
the radii of the balls being the norms of~$x$. Note that the
$K$-orbit $Kx$ consists of all matrices with the same singular values
as~$x$.
\end{example}

\begin{example}\label{expqexcept}%
Now consider the exceptional case $\K=\R$ and $m=n$ of the previous
example, so we have the natural action of $K=SO(n)\times SO(n)$ on
$V=M_n(\R)$.
Here the Weyl chamber $C$ consists of all $x=\diag(x_1,\dots,x_n)\in
\R^n$ with $x_1\ge\cdots\ge x_{n-1}\ge|x_n|$, and
$$P_x\cap C\>=\>\Bigl\{y\in C\colon\sum_{i=1}^k(x_i-y_i)\ge0\
(k=1,\dots,n-1),\ |x_n-y_n|\le\sum_{i=1}^{n-1}(x_i-y_i)\Bigr\}$$
Now the last diagonal entry $x_n$ of $x\in C$ coincides with the
smallest singular value $\sigma_n(x)$ only up to sign. With similar
reasoning as in \ref{expqnonexcept} we conclude
$$\scrO_x\>=\>\bigcap_{k=1}^n\bigl\{y\colon\Vert y\Vert_k\le
\Vert x\Vert_k\bigr\}\cap\Bigl\{y\colon\sigma_n(x)-\sigma_n(y)\le
\Vert x\Vert_{n-1}-\Vert y\Vert_{n-1}\Bigr\}.$$
\end{example}

\begin{rem}\label{ovsso}%
(Orthogonal vs.\ special orthogonal group)
Let $m\ge n$, let $K=SO(m)\times SO(n)$ and $K'=O(m)\times O(n)$. If
$m>n$, it is easy to see that the $K'$-orbit of an arbitrary matrix
$x\in M_{m\times n}(\R)$ coincides with the $K$-orbit of~$x$.
For $m=n$ this is true if $\det(x)=0$, but otherwise $K'x$ is the
union of two distinct $K$-orbits, as one sees from the determinant.
A spectrahedral representation of the $K'$-orbitope $\scrO'_x:=
\conv(K'x)$ was given in \cite[Theorem 4.7]{sss} in the case $m=n$
(and $\scrO'_x$ was called a \emph{Fan orbitope} there). Note that
$\scrO'_x=\conv O(n)$ if $x=I_n$. Moreover, a spectrahedral
representation of the coorbitope $(\conv O(n))^o$ was provided in
\cite{sss} (Corollary 4.9). Both representations can easily be
derived from our discussion of $K$-(co)orbitopes.
\end{rem}

\begin{example}\label{hermact}%
Next consider the actions of the classical compact simple Lie
groups on hermitian matrices. Let $n\ge2$, and let $K=SO(n)$ (case
$\K=\R$) resp.\ $K=SU(n)$ (case $\K=\C$) resp.\ $K=Sp(n)$ (case
$\K=\bbH$).
Let $V=\{x\in M_n(\K)\colon x=x^*\}$, the space of hermitian
$n\times n$-matrices over $\K$, and let $V_0=\{x\in V\colon
\tr(x)=0\}$, where for $\K=\bbH$ the trace condition has to be
replaced by $\trd(x)=0$ (reduced trace). We let $g\in K$ act on
$x\in V$ by $gxg^*$. Clearly $V_0$ is $K$-invariant, and
$V=V_0\oplus\R$ as $K$-modules. The action of $K$ on $V_0$ is an
irreducible polar representation, resulting from the simple Lie
algebra $\g=sl(n,\R)$ resp.\ $\g=su(n,\C)$ resp.\ $\g=sl(n,\bbH)$. We
let $\a\cong\R^{n-1}$ be the space of all real diagonal matrices
$x=(x_1,\dots,x_n)$ with trace zero.
The Weyl chamber is $C=\{x\in\a\colon x_1\ge\cdots\ge x_n\}$, and for
$x=(x_1,\dots,x_n)\in\a$ we have
$$P_x\cap C\>=\>\Bigl\{y\in C\colon\sum_{i=1}^ky_i\le\sum_{i=1}^kx_i\
(i=1,\dots,n)\Bigr\}$$
The $K$-orbit $Kx$ consists of all hermitian matrices with the same
eigenvalues as~$x$ (for $\K=\bbH$ one has to speak of right
eigenvalues instead of eigenvalues \cite{fp}).

In order to describe the $K$-orbitope $\scrO_x$ we replace $x\in V_0$
by $x'=x+cI\in V$, where $c\ge0$ is chosen such that $x'\succeq0$,
i.e.\ $x'$ is positive semidefinite (\emph{psd}). Of course this
doesn't change the orbitope up to an affine-linear isomorphism. So
let $x\in V$ be a psd hermitian matrix. Then clearly $y\succeq0$
holds for every $y\in\scrO_x$,
and the sequence $\sigma(y)=(\sigma_1(y),\dots,\sigma_n(y))$ of
singular values coincides with the sequence of (right) eigenvalues
for these~$y$. So we see that
$$\scrO_x\>=\>\bigl\{y\in V\colon y\succeq0,\ \sigma(y)\nml
\sigma(x)\bigr\},$$
where for nonincreasing sequences $a,\,b\in\R^n$ the majorization
relation $\nml$ is defined by
$$(b_1,\dots,b_n)\>\nml\>(a_1,\dots,a_n)\ :\iff\ \sum_{i=1}^kb_i
\>\le\>\sum_{i=1}^ka_i$$
for $k=1,\dots,n$, with equality for $k=n$ \cite[4.3.41]{hj}.
In terms of Ky Fan norms this says that $\scrO_x$ is the set of all
$y\in V$ with $y\succeq0$ and $\Vert y\Vert_k\le\Vert x\Vert_k$ for
$k=1,\dots,n-1$ and $\Vert y\Vert_n=\Vert x\Vert_n$ (provided that
$x\succeq0$).
\end{example}

\begin{example}\label{unitskewherm}%
Next let the unitary group over $\K$ act on skew-hermitian matrices
over $\K$ by $gxg^*$. For $\K=\C$ this is essentially the action of
$SU(n)$ on hermitian matrices, already considered in \ref{hermact},
since a complex matrix $x$ is skew-hermitian if and only if $ix$ is
hermitian.
For the remaining two cases we have $\K=\R$ and $K=SO(n)$ (with
$n\ge3$),
or $K=\bbH$ and $K=Sp(n)$ (with $n\ge1$),
and $V$ is the space of skew-hermitian ($x+x^*=0$) matrices over $\K$
of size~$n$. This is an irreducible polar representation of $K$,
namely the adjoint action of $K$ on its Lie algebra $V=Lie(K)$ (so
$\g=Lie(K)_\C$ here.)
A maximal abelian subspace can be described as follows. If $\K=\R$
and $K=SO(n)$, put $m:=\lfloor\frac n2\rfloor$ and let $\a$ consist
of all real block matrices $x=\begin{pmatrix}0&\tilde x\\-\tilde x&0
\end{pmatrix}$ where $\tilde x=\diag(x_1,\dots,x_m)$; if $n$ is odd,
an extra row (at the bottom) and column (at the right) of zeros has
to be added. If $\K=\bbH$ and $K=Sp(n)$, let $\a$ consist of all
diagonal matrices $x=(ix_1,\dots,ix_n)$ with $x_1,\dots,x_n\in\R$.

To describe Weyl chamber and orbitopes, let first $\K=\R$ and
$K=SO(n)$. The Weyl chamber $C$ consists of all $x\in\a$ with
$x_1\ge\cdots\ge x_m\ge0$ (case $n$ odd), resp.\ $x_1\ge\cdots\ge
x_{m-1}\ge|x_m|$ (case $n$ even). The description of $P_x\cap C$,
for $x\in C$, is analogous to \ref{expqnonexcept} resp.\
\ref{expqexcept}. The singular values of $x\in C$ are
$x_1,x_1,\dots,x_m,x_m$, with an extra zero if $n$ is odd.
So we get $\scrO_x=\bigcap_{k=1}^m\bigl\{y\in V\colon
\Vert y\Vert_{2k}\le\Vert x\Vert_{2k}\bigr\}$
for $n=2m+1$ odd, and
$$\scrO_x\>=\>\bigcap_{k=1}^m\{y\colon\Vert y\Vert_{2k}\le
\Vert x\Vert_{2k}\}\cap\Bigl\{y\colon\sigma_n(x)-\sigma_n(y)\le
\Vert x\Vert_{n-2}-\Vert y\Vert_{n-2}\Bigr\}$$
for $n=2m$ even. Note that in either case, only the even Ky Fan norms
are needed.

If $\K=\bbH$ and $K=Sp(n)$, the Weyl chamber $C$ consists of all
$x\in\a$ with $x_1\ge\cdots\ge x_m\ge0$, and we find again
$\scrO_x=\bigcap_{k=1}^n\bigl\{y\in V\colon\Vert y\Vert_k\le
\Vert x\Vert_k\bigr\}$ for $x\in V$.
\end{example}

\begin{example}\label{unitcongr}%
There remains the action of the complex unitary group $K=U(n)$ on
$V=sym(n,\C)$ resp.\ $V=so(n,\C)$ (symmetric resp.\ skew-symmetric
complex matrices) by $gxg^t$ ($g\in K$, $x\in V$). Again this is an
irreducible polar representation that arises from the simple Lie
algebra $\g=sp(n,\R)$ (for $V=sym(n,\C)$) resp.\ $\g=so^*(2n)$ (for
$V=so(n,\C)$).

First let $V=sym(n,\C)$. A maximal abelian subspace $\a$ consists of
all real diagonal matrices $x=\diag(x_1,\dots,x_n)$, and the Weyl
chamber is $C=\{x\in\a\colon x_1\ge\cdots\ge x_n\ge0\}$. For $x\in C$
we have $P_x\cap C=\{y\in C\colon y_1+\cdots+y_k\le x_1+\cdots+x_k$
$(k=1,\dots,n)\}$. Since $x_i=\sigma_i(x)$ for $x\in\a$,
we get
$$\scrO_x\>=\>\bigl\{y\in sym(n,\C)\colon\Vert y\Vert_k\le
\Vert x\Vert_k\ (k=1,\dots,n)\bigr\}$$
In the skew-symmetric case $V=so(n,\C)$ let $m=\lfloor\frac n2
\rfloor$. For $n$ even, a maximal abelian subspace $\a$ consists of
all block matrices $x=\begin{pmatrix}0&\tilde x\\-\tilde x&0
\end{pmatrix}$ with $\tilde x=\diag(x_1,\dots,x_m)$ a real diagonal
matrix. For $n$ odd the description is the same, except that one row
(at the bottom) and one column (at the right) of zeros has to be
added. In either case the Weyl chamber $C$ consists of all $x\in\a$
with $x_1\ge\cdots\ge x_m\ge0$. For $x\in\a$ the singular values of
$x$ are $x_1,x_1,\dots,x_m,x_m$, together with an extra zero if $n$
is odd. Once more we therefore find
$$\scrO_x\>=\>\bigl\{y\in so(n,\C)\colon\Vert y\Vert_k\le
\Vert x\Vert_k\ (k=1,\dots,n)\bigr\}$$
The fact that any (skew-) symmetric complex matrix is unitarily
congruent to a real matrix in $\a$ as above is known as Youla's
theorem (see e.g.\ \cite[Theorem 4.4.9]{hj}).%
\end{example}

\begin{rem}\label{doublspectex}%
For all the examples from \ref{expqnonexcept} to \ref{unitcongr},
the following is true: The orbitope $\scrO_x$ is doubly
spectrahedral, provided that all singular values of $x$ are rational
numbers. This follows from Theorem \ref{doublspect}.
\end{rem}

\begin{rem}
As we have seen, most of the classical irreducible polar orbitopes
are intersections of Ky Fan balls of matrices, possibly intersected
with suitable linear spaces of matrices (like (skew-) symmetric or
(skew-) hermitian). It is easy to see that every Ky Fan ball is a
spectrahedron. This gives a second proof of Theorem
\ref{connpolorbsp} in those cases where $\scrO_x$ is an intersection
of such balls:
\end{rem}

\begin{prop}\label{kyfanballspectr}%
Let $\K\in\{\R,\,\C,\,\bbH\}$ and $m\ge n$, let $1\le k\le n$, and
let $\Vert\cdot\Vert_k$ denote the $k$-th Ky Fan norm on
$M_{m\times n}(\K)$ (\ref{kyfanorm}). Then the unit ball
$$B_k\>:=\>\bigl\{x\in M_{m\times n}(\K)\colon\Vert x\Vert_k\le1
\bigr\}$$
is a spectrahedron.
\end{prop}

\begin{proof}
Let $A\in M_N(\C)$ be a complex matrix with eigenvalues
$\theta_1,\dots,\theta_N$. The $k$-th exterior power $\E^kA$ of $A$
is a square matrix of size $\choose Nk$ that depends linearly on $A$,
and whose eigenvalues are the sums $\theta_{i_1}+\cdots+\theta_{i_k}$
with $1\le i_1<\cdots<i_k\le N$.

If $\K=\bbH$, we replace quaternions with complex $2\times2$
matrices, to avoid the problem of defining exterior powers of
quaternion matrices. Let $A\in M_N(\bbH)$ be hermitian ($A=A^*)$,
with (real) right eigenvalues $\theta_1,\dots,\theta_N$. Write
$A=A_1+jA_2$ with $A_1,\,A_2\in M_N(\C)$, and let
$$\tilde A\>=\>\begin{pmatrix}A_1&-\ol A_2\\A_2&\ol A_1
\end{pmatrix}$$
Then $\tilde A$ is a complex hermitian matrix of size $2N$ with
eigenvalues $\theta_1,\theta_1,\dots,\theta_N,\theta_N$. Let us, for
this purpose, define $\E^kA$ to be the complex matrix $\E^k\tilde A$
(of size $\choose{2N}k$).

Now let $\K$ be any of $\R,\,\C,\,\bbH$, let
$x\in M_{m\times n}(\K)$, and let $\wh x\in M_{m+n}(\K)$ be the
hermitian (block) matrix
$$\wh x\ =\ \begin{pmatrix}0&x\\x^*&0\end{pmatrix}$$
The (right) eigenvalues of $\wh x$ are $\pm\sigma_i(x)$,
$i=1,\dots,n$, together with $m-n$ additional zeros. It follows
that $x\in B_k$, i.e.\ $\Vert x\Vert_k\le1$, if and only if all
eigenvalues of $\E^k\wh x$ are $\le1$. In other words, this shows
that $B_k$ is described by the linear matrix inequality
$$B_k\>=\>\bigl\{x\in M_{m\times n}(\K)\colon\E^k\wh x\preceq I
\bigr\}.$$
\end{proof}

\begin{lab}\label{biorb1}%
Finally, we record the cases when the orbitope $\scrO_x$ is a
$K$-biorbitope. First consider the action \ref{hermact} on hermitian
matrices, for $\K=\R,\,\C,\,\bbH$. Up to scaling and translation
there is exactly one biorbitope $\scrO_x$ of this type, namely for
$x=\diag(1,0,\dots,0)$.
The orbit $Kx$ consists of all psd rank one matrices of (reduced)
trace~$1$. Its convex hull has the rank condition removed:
$\scrO_x=\{y\colon y\succeq0$, $\tr(y)=1\}$.
\end{lab}

\begin{lab}\label{biorb2}%
For the remaining actions \ref{expqnonexcept}, \ref{unitskewherm} and
\ref{unitcongr} there exist two essentially different biorbitopes.
When $\K=\R$, we have to exclude the case $m=n$ in
\ref{expqnonexcept} and the case $n$ even in \ref{unitskewherm}.
Indeed, these are the cases when the restricted root system is of
type~$D$ (see Theorem \ref{biorbitop}). Otherwise, the two
biorbitopes are:
\begin{itemize}
\item[(a)]
$x=(1,0,\dots,0)$ and $\scrO_x=\{y\colon\Vert y\Vert_*\le1\}$, the
unit ball in the nuclear norm;
\item[(b)]
$x=(1,\dots,1)$ and $\scrO_x=\{y\colon\Vert y\Vert_1\le1\}$, the unit
ball in the operator norm.
\end{itemize}
In case (b) of the action \ref{expqnonexcept}, the $K$-orbit $Kx$
is the Stiefel manifold $V_n(\K^m)$ of orthonormal $n$-frames in
$\K^m$, at least for $n<m$. We therefore call these orbitopes the
\emph{Stiefel orbitopes}. For $m=n$ and $\K\ne\R$, we get
tautological orbitopes: $\scrO_x$ is the convex hull of $SU(n)$
(case $\K=\C$), resp.\ of $Sp(n)$ (case $\K=\bbH$).
\end{lab}


\end{document}